\documentclass[review,12pt]{elsarticle}
\usepackage{amsmath,amssymb,amsthm}
\usepackage{graphicx}
\usepackage{booktabs}
\usepackage{algorithm}
\usepackage{algpseudocode}
\usepackage[colorlinks=true,linkcolor=black,citecolor=black,urlcolor=blue]{hyperref}

\newtheorem{theorem}{Theorem}
\newtheorem{lemma}{Lemma}
\newtheorem{proposition}{Proposition}
\newtheorem{corollary}{Corollary}
\newtheorem{definition}{Definition}
\theoremstyle{remark}
\newtheorem{remark}{Remark}
\newtheorem{example}{Example}

\newcommand{\Z}{\mathbb{Z}}
\newcommand{\F}{\mathbb{F}}
\newcommand{\Cay}{\mathrm{Cay}}
\newcommand{\diam}{\mathrm{diam}}
\newcommand{\rk}{\mathrm{rank}}

\journal{European Journal of Combinatorics}

\begin{document}
\begin{frontmatter}

\title{Minimal Isometric Embeddings of Graphs into Cayley Graphs of Finite
Abelian Groups}

\author[lesia]{Rigobert Fokam Souop\corref{cor}}
\ead{fokamrigobert@gmail.com}
\author[lesia,saiam]{Laurent Bitjoka}
\cortext[cor]{Corresponding author.}
\address[lesia]{Laboratory of Energy, Signal, Imaging and Automation (LESIA),
University of Ngaound\'er\'e, Ngaound\'er\'e, Cameroon}
\address[saiam]{Laboratory of Scientific Artificial Intelligence and Applied
Mathematics, University of Ngaound\'er\'e, Ngaound\'er\'e, Cameroon}

\begin{abstract}
We study when, and how compactly, a finite connected graph $G$ embeds
isometrically into a Cayley graph of a finite abelian group. The classical
theory of partial cubes answers this question for isometric subgraphs of
hypercubes through the Djokovi\'c--Winkler relation $\theta$; we extend the
question to the full family of abelian Cayley graphs, whose hosts may carry
composite generators and cyclic factors of any order. We introduce an
involutive edge relation $\varphi$, defined by two simultaneous distance
equalities, which coincides with $\theta$ exactly on partial cubes and remains
informative beyond them, together with an oriented relation $\Phi$ for
non-involutive hosts, where generator classes are constrained to be partial
permutations rather than matchings. The central result is a quotient labeling
theorem: for any partition of the edge set into candidate generator classes,
the most generic consistent vertex labeling is the quotient of the free module
on the classes by the lattice of signed cycle--class incidences, computed by
the Smith normal form; the binary case is its reduction modulo two. We prove
that the finest partition always yields an isometric labeling (the Join
Lemmas), that compactifying the resulting universal group is itself an
instance of the same quotient construction---with a sufficient diagonal-fold
criterion, and with non-diagonal sublattices sometimes necessary---and that
the whole construction is algorithmic and certifiable. Worked examples,
carried out in full, include the triangle, the Petersen graph (a certified
embedding into the Clebsch graph of order $16$), the Pappus graph (order
$128$, a $1024$-fold compaction), and the diamond (the order-$6$ octahedron
via a non-diagonal fold). Sharp dimension bounds and an exhaustive census of
small graphs are developed in a companion paper.
\end{abstract}

\begin{keyword}
isometric embedding \sep Cayley graph \sep abelian group \sep partial cube
\sep Djokovi\'c--Winkler relation \sep Smith normal form \sep graph metric
\MSC[2020] 05C12 \sep 05C25 \sep 20K01 \sep 05C50
\end{keyword}
\end{frontmatter}

\section{Introduction}
\label{sec:intro}

A graph $G$ is a \emph{partial cube} if it embeds isometrically into a
hypercube $Q_p=\Cay(\Z_2^p,\{e_1,\dots,e_p\})$, i.e.\ if its vertices admit
binary labels under which graph distance equals Hamming distance. The
question of which graphs admit such labelings is one of the oldest in metric
graph theory: it was raised by Firsov \cite{firsov1965}, and acquired its
applied motivation in the addressing problem of Graham and Pollak
\cite{graham-pollak1971}, whose squashed-cube conjecture was proved by
Winkler \cite{winkler1983}. The structure theory of partial cubes themselves
is classical and complete: Djokovi\'c \cite{djokovic1973} and Winkler
\cite{winkler1984} characterized them through an edge relation $\theta$,
declaring $e=\{u,v\}$ and $f=\{x,y\}$ related when
$d(u,x)+d(v,y)\ne d(u,y)+d(v,x)$, and the resulting coordinatization
underlies isometric dimension, media theory, and a substantial literature in
metric graph theory; see Ovchinnikov's survey \cite{ovchinnikov2008}, the
monographs on graph products \cite{imrich-klavzar2000,hammack2011,ikr2008},
and, for the broader metric context, the survey of Bandelt and Chepoi
\cite{bandelt-chepoi2008} together with the structure theory of metrically
defined graph classes \cite{bandelt1984,bandelt-mulder1986}.

Three further strands of that literature frame the present work. First,
isometric embeddings into \emph{Hamming graphs}---Cartesian products of
complete graphs---were characterized by Winkler \cite{winkler1984} and
Wilkeit \cite{wilkeit1990}, with recognition subsequently brought to near
optimal complexity \cite{aurenhammer1995,imrich-klavzar1996}. Second, the
\emph{canonical isometric embedding} of Graham and Winkler
\cite{graham-winkler1985} represents an arbitrary connected graph in a
Cartesian product of quotient graphs, and Eppstein's \emph{lattice dimension}
\cite{eppstein2005} minimizes integer-lattice representations. Third, the
\emph{scale} and $\ell_1$ theory of Shpectorov \cite{shpectorov1993} and
Deza--Shpectorov \cite{deza-shpectorov1996}, systematized in the monograph of
Deza and Laurent \cite{deza-laurent1997} (see also \cite{laurent1994}),
relaxes isometry to scale-isometry. The subject remains active: recent work
treats Hamming embeddings of weighted graphs \cite{berleant2023,sheridan2021}
and binary stretch embeddings \cite{ebrahimi2025}.

Not every graph is a partial cube---odd cycles, the complete graphs $K_n$ for
$n\ge 3$, and the Petersen graph are not---and for such graphs the hypercube
is simply the wrong host. All the hosts above, however, are special
\emph{abelian Cayley graphs}: the hypercube is
$\Cay(\Z_2^p,\{e_1,\dots,e_p\})$, and Hamming graphs are products of complete
circulants. This paper asks the natural broader question:

\begin{quote}
\emph{Into which Cayley graph of a finite abelian group does a given
connected graph embed isometrically, and how small can that host be made?}
\end{quote}

The family of abelian Cayley graphs is far richer than the family of
hypercubes. It contains the cycles $C_n=\Cay(\Z_n,\{\pm 1\})$, tori and grids,
circulants, cocktail-party and Hamming graphs, and the Clebsch graph
$\Cay(\Z_2^4,S)$ with a weight-four composite generator. Allowing
\emph{composite} generators (group elements that are sums of several basic
directions) and higher-order cyclic factors $\Z_m$ makes isometric hosts
available for \emph{every} connected graph, and often makes them dramatically
smaller than the naive bound: the triangle $K_3$, not a partial cube, is
$\Cay(\Z_3,\{1,2\})$ itself; the Petersen graph embeds isometrically into the
Clebsch graph of order $16$, far below its $1024$-vertex naive binary host.

\paragraph{What is standard here and what is new}
We state this plainly, because the two ingredients of our construction have
very different pedigrees. The algebraic engine---given a partition of the
edges, the universal consistent labeling is the cokernel of the (signed)
cycle--class incidence map, computed by the Smith normal form---is a
classical computation: it is the determination of the first homology of the
graph with coefficients twisted by the partition, and the Smith normal form
is the standard tool for such cokernels (see Stanley's survey
\cite{stanley2016} for its combinatorial uses, and
\cite[Ch.~14]{godsil-royle2001} for the cycle space). We claim no novelty for
the engine and Section~\ref{sec:quotient} says so explicitly
(Remark~\ref{rem:homology}). What we believe is new is the \emph{metric}
theory built on top of it: the relation $\varphi$
(Definition~\ref{def:phi}), defined by two simultaneous distance equalities,
which coincides with $\theta$ exactly on partial cubes
(Theorem~\ref{thm:phitheta}) and, unlike $\theta$, remains a useful
compaction guide beyond them; the observation that in a non-involutive host
generator classes are partial permutations rather than matchings
(Proposition~\ref{prop:partialperm}), which repairs a natural but false
generalization of the partial-cube paradigm; the Join Lemmas
(Lemmas~\ref{lem:join2} and~\ref{lem:joinZ}), which prove that the finest
partition is always isometric and thereby convert the algebraic engine into
an embedding machine with a universal guarantee; the analysis of
compactification as a second application of the same quotient construction,
including a sufficient diagonal-fold criterion (Theorem~\ref{thm:fold}) and
the fact---witnessed by the diamond and verified exhaustively---that
non-diagonal sublattices are sometimes necessary (Theorem~\ref{thm:sublattice});
and the certified algorithmic pipeline of Section~\ref{sec:algorithm}.

\paragraph{Relation to the Graham--Winkler canonical embedding}
The canonical embedding \cite{graham-winkler1985} represents $G$ in a product
of quotient graphs indexed by the $\theta^\ast$-classes (the transitive
closure of $\theta$); it is universal among isometric embeddings into
Cartesian products. Our construction can be read as a group-structured
analogue: where Graham--Winkler quotient by $\theta^\ast$-classes and take an
unrestricted product, we partition edges into candidate \emph{generator}
classes and take the universal \emph{abelian group} compatible with the cycle
structure. The two constructions coincide in spirit on partial cubes, where
$\varphi=\theta$ and the coordinates are cuts; beyond partial cubes they
diverge, since a group host imposes relations (torsion, composite generators)
that a free product does not. A precise comparison theorem is an interesting
open direction; we do not attempt it here.

\paragraph{Contributions and organization}
Section~\ref{sec:prelim} fixes notation. Section~\ref{sec:relations} develops
the two relations: the involutive $\varphi$ with its properties and the proof
that it coincides with $\theta$ on partial cubes (Theorem~\ref{thm:phitheta},
both directions in full), and the oriented $\Phi$ with the
partial-permutation constraint (Proposition~\ref{prop:partialperm}).
Section~\ref{sec:quotient} proves the quotient labeling theorems: the cocycle
conditions (Theorems~\ref{thm:cocycle2} and~\ref{thm:cocycleZ}), the binary
and general quotient theorems (Theorems~\ref{thm:quotient2}
and~\ref{thm:quotientZ}), and the Join Lemmas, whose binary case we prove via
a geodesic-independence argument that also repairs a gap in an earlier draft
of this work. Section~\ref{sec:compact} treats compactification: the
sufficient diagonal-fold criterion (Theorem~\ref{thm:fold}), the necessity of
non-diagonal sublattices (Theorem~\ref{thm:sublattice}, with the diamond data
verified exhaustively and reported in full), and the resulting existence
statement (Theorem~\ref{thm:existence}), which we frame as the classical
spanning-tree baseline rather than as a new result.
Section~\ref{sec:algorithm} presents the algorithm in the body of the paper,
with its design rationale and a candid complexity analysis
(Theorem~\ref{thm:complexity} and the discussion following it): the
certification step is exponential in the binary dimension in the worst case,
and we say so. Section~\ref{sec:examples} works four examples in full,
displaying the actual $\varphi$-classes and cycle--class parity matrices for
the Petersen and Pappus graphs so that every number is reproducible.
Section~\ref{sec:conclusion} concludes.

Sharp lower bounds and optimality---including the injectivity bound that
makes the Petersen embedding optimal and the exact star dimension---together
with an exhaustive census of all $995$ connected graphs on at most seven
vertices \cite{read-wilson1998}, are developed in the companion paper \cite{fokam-p2}, of which the
present construction is the foundation. Our motivation is partly applied:
isometric abelian hosts equip a graph with the harmonic analysis of a finite
abelian group, connecting graph signal processing
\cite{shuman2013,sandryhaila-moura2013,ortega2018} and spectral methods
\cite{chung1997,hammond2011} to classical Fourier theory; that direction is
developed in two further companion papers \cite{fokam-p3,fokam-p4} and in the
first author's dissertation \cite{fokam-diss}.

\section{Preliminaries}
\label{sec:prelim}

Throughout, $G=(V,E)$ is a finite connected graph with $n=|V|$, $m=|E|$,
shortest-path metric $d=d_G$, and diameter $\diam(G)$. We write $c=m-n+1$ for
the cycle rank (first Betti number) of $G$. A finite abelian group is denoted
$\Gamma$; its identity is $0$.

\begin{definition}[Cayley graph]\label{def:cayley}
For a finite abelian group $\Gamma$ and a generating set
$S\subseteq\Gamma\setminus\{0\}$ with $S=-S$, the Cayley graph
$\Cay(\Gamma,S)$ has vertex set $\Gamma$ and an edge $\{g,g+s\}$ for each
$g\in\Gamma$ and $s\in S$. It is connected (as $S$ generates),
vertex-transitive, and its graph metric is the word metric, translation
invariant: $d_\Gamma(x,y)=d_\Gamma(0,y-x)=|y-x|_S$, the least number of
generators summing to $y-x$.
\end{definition}

\begin{definition}[Isometric embedding]\label{def:isometric}
An isometric embedding of $G$ into $\Cay(\Gamma,S)$ is an injection
$\phi\colon V\to\Gamma$ with $d_\Gamma(\phi(u),\phi(v))=d_G(u,v)$ for all
$u,v\in V$. An edge $\{u,v\}$ then satisfies $\phi(v)-\phi(u)=\pm s$ for some
$s\in S$; we say the edge \emph{carries} the generator $s$. We call
$N=|\Gamma|$ the host order and seek to minimize it.
\end{definition}

The \emph{naive construction} roots a spanning tree at $r$, gives each of the
$n-1$ tree edges its own coordinate in $\Z_2^{n-1}$, and labels $v$ by the
sum of coordinates along the tree path $r\to v$; non-tree edges receive
composite labels. This construction is classical (it is implicit already in
the earliest work on Boolean-cube embeddings \cite{firsov1965}) and is always
isometric---we re-derive it as the finest binary quotient in
Corollary~\ref{cor:naive}---but wasteful: a host of order $2^{n-1}$. The
point of the theory is to merge edges into shared generators wherever the
metric permits, shrinking $k$ (hence $N=2^k$ in the binary case) or replacing
binary directions by cyclic factors $\Z_m$.

\section{Two relations generalizing Djokovi\'c--Winkler}
\label{sec:relations}

\subsection{The involutive relation $\varphi$}

\begin{definition}[$\varphi$ relation]\label{def:phi}
For edges $e=\{u,v\}$ and $f=\{x,y\}$ of $G$, write $e\;\varphi\;f$ if
\[
d(u,x)=d(v,y)\quad\text{and}\quad d(u,y)=d(v,x).
\]
\end{definition}

The relation captures \emph{metric parallelism}: the four inter-endpoint
distances obey the symmetry of opposite sides of a parallelogram. Edges
carrying the same involutive generator in a binary Cayley graph satisfy
precisely this symmetry (Theorem~\ref{thm:phitheta} below makes this exact on
partial cubes).

\begin{remark}[Position with respect to known relations]\label{rem:phiposition}
The Djokovi\'c--Winkler relation declares $e\;\theta\;f$ when
$d(u,x)+d(v,y)\ne d(u,y)+d(v,x)$; $\varphi$ instead demands the two
\emph{simultaneous equalities} of Definition~\ref{def:phi}, which is neither
implied by nor implies $\theta$ in general. Distance-equality conditions
between edge pairs pervade the literature on isometric embeddings into
products \cite{wilkeit1990,bandelt-chepoi2008}, and we make no claim of
priority for the general idea of comparing endpoint distances; we have not,
however, found the specific conjunction of Definition~\ref{def:phi} studied
under another name, and we would welcome a pointer. What we do claim is its
behaviour: Theorem~\ref{thm:phitheta} shows $\varphi$ recovers $\theta$
exactly on partial cubes, while beyond partial cubes---where $\theta$
degenerates (on the Petersen graph the whole edge set is a single
$\theta$-class)---$\varphi$ retains fine structure (five classes of size
three on Petersen) that drives compact embeddings. We are also explicit about
its limitation: $\varphi$ is not transitive
(Theorem~\ref{thm:phiprops}(ii)), so it cannot by itself define a
coordinatization; it is a \emph{guide} for proposing generator classes, and
the exact consistency burden is carried by the quotient machinery of
Section~\ref{sec:quotient}. Section~\ref{sec:examples} quantifies the point:
on the Pappus graph only $2$ of the $15$ partitions of $E$ into
pairwise-$\varphi$ triples are isometric, so $\varphi$-compatibility is
genuinely weaker than embeddability, and certification cannot be dispensed
with.
\end{remark}

\begin{lemma}[Well-definedness]\label{lem:welldef}
Definition~\ref{def:phi} is independent of the chosen orientations of $e,f$.
\end{lemma}

\begin{proof}
Swapping $x\leftrightarrow y$ exchanges the two required equalities with each
other, preserving their conjunction; swapping $u\leftrightarrow v$ likewise.
All four orientation choices impose the same pair of conditions.
\end{proof}

\begin{theorem}[Properties of $\varphi$]\label{thm:phiprops}
For any connected graph $G$:
\begin{itemize}
\item[(i)] $\varphi$ is reflexive and symmetric;
\item[(ii)] $\varphi$ is in general not transitive;
\item[(iii)] incident edges are never $\varphi$-related; hence every set of
pairwise $\varphi$-related edges is a matching;
\item[(iv)] two distinct edges on a common shortest path are never
$\varphi$-related;
\item[(v)] in the even cycle $C_{2\nu}$ the maximal pairwise-$\varphi$ sets
are exactly the $\nu$ antipodal pairs;
\item[(vi)] in the odd cycle $C_{2\nu+1}$ all $\varphi$-classes are
singletons.
\end{itemize}
\end{theorem}

\begin{proof}
(i) $d(u,u)=d(v,v)=0$ and $d(u,v)=d(v,u)$; symmetry is immediate from the
symmetry of the conditions in $\{e,f\}$.

(iii) If $e=\{u,v\}$, $f=\{u,y\}$ with $v\ne y$, then $d(u,u)=0<d(v,y)$, so
$e\;\not\!\varphi\;f$; a pairwise-$\varphi$ set is therefore a matching.

(iv) Orient a common shortest path so the order along it is $u,v,\dots,x,y$.
Sub-paths of shortest paths are shortest, so $d(u,x)=1+d(v,x)$ and
$d(v,y)=d(v,x)+1$, giving the first equality of Definition~\ref{def:phi};
but $d(u,y)=d(v,x)+2\ne d(v,x)$, so the second fails, and by
Lemma~\ref{lem:welldef} no orientation repairs it.

(v)--(vi) In $C_m$ with cyclic distance
$\delta(t)=\min(t\bmod m,\,m-(t\bmod m))$, take non-incident
$e=\{v_a,v_{a+1}\}$, $f=\{v_b,v_{b+1}\}$ and $t=b-a\bmod m$. Forward,
$d(v_a,v_b)=\delta(t)=d(v_{a+1},v_{b+1})$, so the test reduces to
$\delta(t+1)=\delta(t-1)$, i.e.\ $t+1\equiv\pm(t-1)\pmod m$. The $+$ sign
gives $2\equiv 0$, impossible for $m\ge 3$; the $-$ sign gives
$2t\equiv 0$. For odd $m$ this forces $t\equiv 0$ (singletons); for $m=2\nu$
it forces $t\equiv 0$ or $t\equiv\nu$, the latter being the antipodal pair.

(ii) is witnessed by $K_{2,3}$ (Figure~\ref{fig:k23}): with parts
$\{u_1,u_2\}$, $\{v_1,v_2,v_3\}$, the edges $a=u_1v_1$ and $b=u_2v_2$ are
$\varphi$-related, as are $b$ and $c=u_1v_3$, but $a,c$ share $u_1$, so
$a\;\not\!\varphi\;c$ by~(iii).
\end{proof}

\begin{figure}[t]
\centering
\includegraphics[width=.5\linewidth]{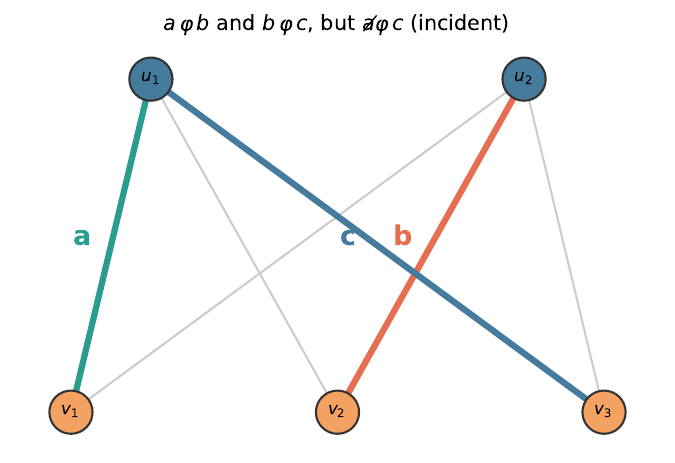}
\caption{Non-transitivity of $\varphi$ in $K_{2,3}$: $a\,\varphi\,b$ and
$b\,\varphi\,c$, yet $a\,\not\!\varphi\,c$ because $a,c$ are incident. The
relation captures local parallelisms that can conflict globally; the quotient
construction of Section~\ref{sec:quotient} resolves the conflicts exactly.}
\label{fig:k23}
\end{figure}

The two relations $\varphi$ and $\theta$ differ markedly away from partial
cubes: on odd cycles $\theta$ is non-transitive while $\varphi$ is
(vacuously) an equivalence relation; and on the Petersen graph the whole edge
set is a single $\theta$-class (it is not a partial cube), whereas $\varphi$
has five classes of size three, one per parallel matching, which drive the
compact embedding of Example~\ref{ex:petersen}. On partial cubes, however,
the two agree, and this is a theorem in both directions.

\begin{theorem}[$\varphi$ on partial cubes]\label{thm:phitheta}
If $G$ is a partial cube, then $\varphi$ coincides with $\theta$, is an
equivalence relation on $E(G)$, and its classes are exactly the Djokovi\'c
cuts. Conversely, if $G$ is bipartite, $\varphi$ is an equivalence relation,
and each $\varphi$-class induces an edge cut with convex sides, then $G$ is a
partial cube with $\theta$-classes equal to its $\varphi$-classes.
\end{theorem}

\begin{proof}
\emph{Forward direction.} Fix an isometric $\lambda\colon V(G)\to Q_p$, so
$d_G$ is Hamming distance and each edge flips one coordinate. Let $e$ flip
coordinate $i$, $f$ flip $j$, and $z=\lambda(u)\oplus\lambda(x)$. Then
$d(u,x)=\mathrm{wt}(z)$, $d(v,y)=\mathrm{wt}(z\oplus e_i\oplus e_j)$,
$d(u,y)=\mathrm{wt}(z\oplus e_j)$, $d(v,x)=\mathrm{wt}(z\oplus e_i)$. If
$i=j$, the first two and the last two agree for all $z$, and the
side-respecting orientation makes both $\varphi$ equalities hold. If
$i\ne j$, the first equality holds iff $z_i\ne z_j$ and the second iff
$z_i=z_j$, mutually exclusive; so edges of distinct cuts are never
$\varphi$-related. Hence the $\varphi$-classes are the coordinate cuts,
which are the $\theta$-classes of a partial cube \cite{djokovic1973},
$\varphi=\theta$, and transitivity follows.

\emph{Converse direction.} Assume $G$ bipartite, $\varphi$ an equivalence
relation, and every $\varphi$-class $F$ an edge cut $\{A_F,B_F\}$ with both
sides convex. We verify Djokovi\'c's criterion \cite{djokovic1973}: a
bipartite graph is a partial cube if and only if, for every edge
$\{u,v\}$, the sets
$W_{uv}=\{w:d(w,u)<d(w,v)\}$ and $W_{vu}=\{w:d(w,v)<d(w,u)\}$ are convex.
Fix an edge $e=\{u,v\}$, let $F$ be its $\varphi$-class and $\{A_F,B_F\}$ the
associated cut with $u\in A_F$, $v\in B_F$. We claim $W_{uv}=A_F$ and
$W_{vu}=B_F$, which yields convexity of $W_{uv}$ and $W_{vu}$ from the
hypothesis and completes the proof.

First, $A_F\subseteq W_{uv}$. Let $w\in A_F$ and let $P$ be a shortest
$w$--$v$ path. Since $w\in A_F$ and $v\in B_F$, $P$ crosses the cut $F$ at
some edge $f=\{x,y\}$ with $x\in A_F$, $y\in B_F$. Because $f\in F$ and
$\varphi$-classes consist of $\varphi$-related edges, $e\;\varphi\;f$, so
$d(u,x)=d(v,y)$ and $d(u,y)=d(v,x)$. Both sides of the cut are convex, so the
sub-path of $P$ from $w$ to $x$ stays in $A_F$ and the sub-path from $y$ to
$v$ stays in $B_F$; in particular $P$ crosses $F$ exactly once. Now
\[
d(w,v)=d(w,x)+1+d(y,v)=d(w,x)+1+d(x,u)\ge d(w,u)+1>d(w,u),
\]
using $d(y,v)=d(x,u)$ from the $\varphi$ equalities and the triangle
inequality. Hence $w\in W_{uv}$. Symmetrically $B_F\subseteq W_{vu}$. Since
$G$ is bipartite, $V=W_{uv}\,\dot\cup\,W_{vu}$ (no vertex is equidistant from
the two ends of an edge), and since $V=A_F\,\dot\cup\,B_F$ as well, the two
inclusions force $W_{uv}=A_F$ and $W_{vu}=B_F$, as claimed. Djokovi\'c's
criterion now applies, $G$ is a partial cube, and its $\theta$-classes are
the cuts $\{A_F,B_F\}$, i.e.\ the $\varphi$-classes.
\end{proof}

\begin{remark}\label{rem:notmin}
Theorem~\ref{thm:phitheta} does not say the isometric dimension is the
minimum over the larger family $\Cay(\Z_2^k,S)$: composite generators can do
strictly better even on partial cubes. The companion paper \cite{fokam-p2}
exhibits stars---the simplest partial cubes---whose minimum binary dimension
is $\lceil\log_2 q\rceil+1$, exponentially below the isometric dimension;
the underlying extremal question is that of maximum sum-free sets in
$\Z_2^k$ \cite{rhemtulla-street1970,green-ruzsa2005}.
\end{remark}

\subsection{Oriented partitions and the relation $\Phi$}

Over $\Z_2^k$ every generator is an involution ($g=-g$) and edges need no
orientation. Over a general abelian group, traversing an edge forward adds
$g$ and backward adds $-g\ne g$, so generator assignments must carry
orientations.

\begin{definition}[Oriented partition]\label{def:oriented}
An \emph{oriented partition} of $G$ is a partition
$\mathcal{P}=\{F_1,\dots,F_t\}$ of $E(G)$ together with a chosen direction
$u\to v$ for each edge, with the intended semantics that all edges of $F_j$
carry the same generator $g_j$ added in the forward direction.
\end{definition}

The following observation is elementary---it is an immediate consequence of
injectivity, and we state it as a proposition rather than a theorem for that
reason---but it corrects a natural error: the matching constraint of the
hypercube setting does \emph{not} survive the generalization.

\begin{proposition}[Partial-permutation constraint]\label{prop:partialperm}
In any isometric embedding $\phi\colon V(G)\to\Cay(\Gamma,S)$, the set of
edges carrying a fixed generator $g$, with forward orientations, has
in-degree at most $1$ and out-degree at most $1$ at every vertex.
Equivalently, every generator class is a partial permutation: a disjoint
union of directed paths and directed cycles.
\end{proposition}

\begin{proof}
Two forward edges out of $v$ would give $\phi(w_1)=\phi(v)+g=\phi(w_2)$ with
$w_1\ne w_2$, contradicting injectivity; dually for two forward edges into
$v$. A digraph with all in- and out-degrees at most $1$ is a disjoint union
of directed paths and cycles.
\end{proof}

\begin{remark}[Matching is the involutive special case]\label{rem:matching}
It is tempting to require every class to be a matching, as the binary case
suggests. That is false in general: in $K_3=\Cay(\Z_3,\{1,2\})$ all three
edges carry the same generator $g=1$, oriented around the cycle
(Figure~\ref{fig:k3})---pairwise incident, not a matching, yet a valid
directed $3$-cycle class. The matching constraint is exactly the involutive
case $2g=0$ of Proposition~\ref{prop:partialperm}; imposing it would forbid
the directed-cycle classes that produce the odd cyclic factors
$\Z_{2k+1}$, e.g.\ forcing $K_3\hookrightarrow\Z_4$ where $\Z_3$ is optimal.
\end{remark}

\begin{figure}[t]
\centering
\includegraphics[width=.36\linewidth]{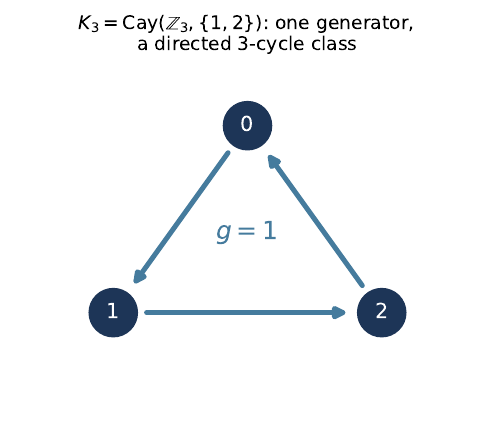}
\caption{The corrected class constraint. In $K_3=\Cay(\Z_3,\{1,2\})$ one
generator carries all three edges as a directed cycle. Classes are partial
permutations, not matchings; matchings are the involutive case $2g=0$.}
\label{fig:k3}
\end{figure}

\begin{definition}[Oriented $\Phi$-test]\label{def:Phi}
Ordered edges $(u\to v)$ and $(x\to y)$ pass the oriented $\Phi$-test if
$d(u,x)=d(v,y)$. Two undirected edges are $\Phi$-related if some choice of
orientations passes the test.
\end{definition}

\begin{proposition}[Necessity]\label{prop:necessity}
If $(u\to v)$ and $(x\to y)$ carry the same generator $g$ in an isometric
embedding into $\Cay(\Gamma,S)$, then $d(u,x)=d(v,y)$, and moreover
$|d(u,y)-d(u,x)|\le 1$ and $|d(v,x)-d(u,x)|\le 1$.
\end{proposition}

\begin{proof}
Write $z=\phi(x)-\phi(u)$ and $|\cdot|$ for the word norm. Then
$\phi(y)-\phi(v)=(\phi(x)+g)-(\phi(u)+g)=z$, so $d(v,y)=|z|=d(u,x)$. Also
$\phi(y)-\phi(u)=z+g$ and $\phi(x)-\phi(v)=z-g$, and the word norm changes by
at most $1$ under addition of a single generator.
\end{proof}

For an involution $g=-g$ both equalities $d(u,x)=d(v,y)$ and $d(u,y)=d(v,x)$
hold, recovering $\varphi$; for $g\ne-g$ only the first survives, as an OR
over the two orientations. Thus $\Phi$ is a necessary filter, deliberately
weaker than $\varphi$: its role in this paper is to prune the search space of
oriented partitions (Section~\ref{sec:algorithm}), and the full consistency
burden is carried by the exact core of the next section.

\section{The quotient labeling theorem}
\label{sec:quotient}

This section is the theoretical core. Given a partition of $E(G)$ into
candidate generator classes, we determine exactly when the induced labeling
is consistent, and show that the most generic consistent labeling is computed
by linear algebra: a quotient of the free module on the classes by the
lattice of signed cycle--class incidences. We treat the binary case first for
transparency, then the general case, of which the binary case is the
reduction modulo~$2$.

\begin{remark}[The engine is homological, and we say so]\label{rem:homology}
Readers from algebraic combinatorics will recognize everything in this
section as the computation of a cokernel of a boundary-type map: the
consistency condition is that the class assignment vanish on the cycle space,
and the universal solution is the quotient of the class module by the image
of the cycle lattice---over $\F_2$ this is standard cycle-space linear
algebra \cite[Ch.~14]{godsil-royle2001}, and over $\Z$ the Smith normal form
is the standard constructive tool \cite{stanley2016}. We make no originality
claim for this machinery. The content of the section lies in the
\emph{metric} statements attached to it: part~(iii) of
Theorems~\ref{thm:quotient2} and~\ref{thm:quotientZ} (quotient labelings
never stretch), and above all the Join Lemmas (Lemmas~\ref{lem:join2}
and~\ref{lem:joinZ}), which show that the finest partition's quotient is
isometric---the statement that turns a homology computation into an
embedding theorem.
\end{remark}

\subsection{The binary ($\F_2$) case}

Given a partition $\mathcal{P}=\{F_1,\dots,F_t\}$ of $E(G)$, assign
generators $g_j\in\Z_2^k$ to the classes and label
\begin{equation}\label{eq:label}
\lambda(v)=\sum_{e\in P(r,v)} g_{\mathrm{class}(e)}\quad(\text{sum in }\Z_2^k),
\end{equation}
along a path $P(r,v)$ from a fixed root $r$. We ask when $\lambda$ is well
defined, i.e.\ independent of the chosen path.

\begin{theorem}[Cocycle condition]\label{thm:cocycle2}
The labeling~\eqref{eq:label} is well defined iff for every cycle
$C=(e_1,\dots,e_\ell)$ of $G$,
\begin{equation}\label{eq:star}
\sum_{i=1}^{\ell} g_{\mathrm{class}(e_i)}=0\in\Z_2^k. \tag{$\star$}
\end{equation}
It suffices to verify~\eqref{eq:star} on any cycle basis of $G$.
\end{theorem}

\begin{proof}
Two $r$--$v$ paths give equal sums iff the sum over their symmetric
difference vanishes; that symmetric difference is an edge-disjoint union of
cycles, so~\eqref{eq:star} for all cycles implies path-independence, and a
violating cycle yields two $r$--$v$ paths with different sums. The cycle
space over $\F_2$ is spanned by any cycle basis
\cite[Ch.~14]{godsil-royle2001} and~\eqref{eq:star} is linear in the cycle,
so a basis suffices.
\end{proof}

\begin{remark}[Why the cut paradigm works on partial cubes]\label{rem:cuts}
If a class $F_i$ is an edge cut, every cycle crosses it an even number of
times, so $F_i$ contributes $0$ to every instance of~\eqref{eq:star}
regardless of its generator. When classes are not cuts, some cycle crosses a
class an odd number of times, and naive breadth-first XOR propagation becomes
order-dependent---the root cause of failure of cut-based heuristics on
non-partial-cubes.
\end{remark}

Fix a cycle basis $B_1,\dots,B_c$ and define the cycle--class parity matrix
$A\in\F_2^{c\times t}$ by $A[i,j]=|F_j\cap E(B_i)|\bmod 2$.
Condition~\eqref{eq:star} on the basis reads
$A\,(g_1\mid\cdots\mid g_t)^{\!\top}=0$ over $\F_2$.

\begin{theorem}[Quotient Labeling Theorem, binary case]\label{thm:quotient2}
Let $\mathcal{P}$ have $t$ classes and $\rho=\rk_{\F_2}(A)$. Then:
\begin{itemize}
\item[(i)] the most generic consistent assignment has dimension $k=t-\rho$,
given by $g_j=\pi(e_j)$, the images of the standard basis of $\Z_2^t$ under
the quotient map $\pi\colon\Z_2^t\to\Z_2^t/\mathrm{rowspan}(A)\cong\Z_2^k$;
every other consistent assignment is a linear image of it;
\item[(ii)] the labeling~\eqref{eq:label} under this assignment is
conflict-free;
\item[(iii)] with $S=\{g_j:g_j\ne 0\}$, every $G$-path of length $\ell$ maps
to a Cayley walk of length $\ell$, so
$d_{\Cay}(\lambda(u),\lambda(v))\le d_G(u,v)$ for all $u,v$: the only
possible failure of isometry is a shortcut, never a stretch.
\end{itemize}
\end{theorem}

\begin{proof}
(i) The solution set of $Ax=0$ is a subspace of dimension $t-\rho$; any
consistent assignment in any $\Z_2^{k'}$ defines a linear map
$\Z_2^t\to\Z_2^{k'}$ vanishing on $\mathrm{rowspan}(A)$, which factors
through $\pi$, so $\pi$ is universal. (ii) By Theorem~\ref{thm:cocycle2},
the images satisfy~\eqref{eq:star} on a basis, hence on every cycle.
(iii) Consecutive labels along a path differ by the class generator of the
traversed edge, an element of $S\cup\{0\}$; zero generators only shorten the
walk, and minimizing over paths gives the inequality.
\end{proof}

By Theorem~\ref{thm:quotient2}(iii) only $d_{\Cay}\ge d_G$ can fail. The next
lemma shows it never fails for the finest partition. Its proof rests on a
geodesic-independence property of binary Cayley words which is useful in its
own right (it also underlies the dimension lower bounds of the companion
paper \cite{fokam-p2}).

\begin{lemma}[Geodesic independence]\label{lem:geodindep}
If $s_1,\dots,s_\ell\in S$ are the letters of a geodesic word in
$\Cay(\Z_2^k,S)$, i.e.\ $|s_1+\dots+s_\ell|_S=\ell$, then the $s_i$ are
pairwise distinct, and indeed linearly independent over $\F_2$.
\end{lemma}

\begin{proof}
First, every sub-multiset of a geodesic word is geodesic: if
$T\subseteq\{1,\dots,\ell\}$ and $\bigl|\sum_{i\in T}s_i\bigr|_S=\ell'<|T|$,
witnessed by letters $w_1,\dots,w_{\ell'}$, then---the group being abelian,
letters may be freely reordered---replacing the letters indexed by $T$ with
$w_1,\dots,w_{\ell'}$ produces a word of length $\ell-|T|+\ell'<\ell$ for the
same sum, contradicting geodesicity. Hence
$\bigl|\sum_{i\in T}s_i\bigr|_S=|T|$ for all $T$. If some nonempty subset
summed to zero (in particular, if two letters coincided, since $2s=0$), its
word norm would be $0<|T|$, a contradiction; and distinct subset sums for all
subsets is linear independence over $\F_2$.
\end{proof}

\begin{lemma}[Join Lemma]\label{lem:join2}
Let $\mathcal{P}_0$ be the all-singleton partition. Then its quotient
labeling is isometric: $d_{\Cay}(\lambda(u),\lambda(v))=d_G(u,v)$ for all
$u,v$.
\end{lemma}

\begin{proof}
Identify $\Z_2^m$ with the edge space $\F_2^{E}$ and
$\mathrm{rowspan}(A)$ with the cycle space $\mathcal{Z}(G)$; the quotient map
is $\pi\colon\F_2^{E}\to\F_2^{E}/\mathcal{Z}(G)$, and
$\lambda(u)\oplus\lambda(v)=\pi(\chi_P)$ for the characteristic vector
$\chi_P$ of any $u$--$v$ path $P$.

Let $\ell=d_{\Cay}(\lambda(u),\lambda(v))$ and let $s_1,\dots,s_\ell$ be a
geodesic word from $\lambda(u)$ to $\lambda(v)$. By
Lemma~\ref{lem:geodindep} the letters $s_1,\dots,s_\ell$ are pairwise
distinct elements of $S$. Each letter is the image $\pi(\chi_e)$ of a
distinct edge $e$ (for the all-singleton partition, $S$ consists of the
images of the single edges, and distinct letters lift to distinct edges after
fixing, for each letter, one edge representing it); let
$F\subseteq E$ be a set of $\ell$ such representative edges, so that
$\pi(\chi_F)=\sum_i s_i=\lambda(u)\oplus\lambda(v)=\pi(\chi_P)$.

Then $\chi_F\oplus\chi_P\in\mathcal{Z}(G)$, so over $\F_2$ the boundaries
agree: $\partial F=\partial P=\{u,v\}$. In the spanning subgraph $(V,F)$,
every connected component contains an even number of odd-degree vertices;
the component containing $u$ therefore also contains $v$, hence contains a
$u$--$v$ path, and so $\ell=|F|\ge d_G(u,v)$. Combined with
Theorem~\ref{thm:quotient2}(iii), equality holds.
\end{proof}

\begin{remark}\label{rem:gapfix}
An earlier draft of this proof asserted directly that ``a geodesic word uses
distinct generators, hence corresponds to an edge set of size $\ell$''
without justifying distinctness; a referee rightly observed that the
assertion needs an argument when composite generators are present. The
splicing argument of Lemma~\ref{lem:geodindep} supplies exactly that
justification, and we thank the referee for the observation.
\end{remark}

\begin{corollary}[The naive embedding is the finest quotient]\label{cor:naive}
For $\mathcal{P}_0$, $A$ is the cycle-basis~$\times$~edge incidence matrix,
of rank $c=m-n+1$, so $k=m-(m-n+1)=n-1$; the free coordinates may be taken on
the edges of any spanning tree, recovering the spanning-tree labeling, which
is isometric by Lemma~\ref{lem:join2}. In particular every connected graph
embeds isometrically into $\Z_2^{\,n-1}$, and the bound $k\le n-1$ is
universal.
\end{corollary}

\subsection{The general abelian ($\Z$) case}

Over a general abelian group we use signed sums. For an oriented partition
with generators $g_j\in\Gamma$, the labeling is
$\phi(v)=\sum_{e\in P(r,v)}\pm g_{\mathrm{class}(e)}$, the sign $+$ for
forward traversal.

\begin{theorem}[$\Z$-cocycle condition]\label{thm:cocycleZ}
The labeling $\phi$ is well defined iff for every cycle of $G$, traversed in
a fixed rotational sense, the signed sum of class generators vanishes; and it
suffices to verify this on a cycle basis.
\end{theorem}

\begin{proof}
Identical to Theorem~\ref{thm:cocycle2} with signed sums: two $r$--$v$ paths
differ by an element of the integer cycle space, signed cycle sums are
$\Z$-linear in the cycle, and a cycle basis spans the cycle space over~$\Z$.
\end{proof}

For a cycle basis $B_1,\dots,B_c$ define the signed cycle--class matrix
$A\in\Z^{c\times t}$ by $A[i,j]=$ net signed number of crossings of $F_j$ by
$B_i$ (forward $+1$, backward $-1$).

\begin{theorem}[$\Z$-Quotient Theorem]\label{thm:quotientZ}
Let $(\mathcal{P},\omega)$ be an oriented partition of the connected graph
$G$ with signed matrix $A$. Then:
\begin{itemize}
\item[(i)] the most generic consistent generator assignment takes values in
the finitely generated abelian group
$\Gamma_{\mathrm{univ}}=\Z^t/\mathrm{rowlattice}(A)$, with $g_j$ the image of
the $j$-th standard basis vector; every consistent assignment in any abelian
group is a homomorphic image of it;
\item[(ii)] the Smith normal form computes it: if
$U A^{\!\top} W=\mathrm{diag}(d_1,\dots,d_\rho,0,\dots)$ with $U,W$
unimodular, then
\[
\Gamma_{\mathrm{univ}}\cong\Z_{d_1}\times\cdots\times\Z_{d_\rho}\times
\Z^{\,t-\rho}
\]
(factors with $d_i=1$ trivial), and the coordinates of $g_j$ are read from
the $j$-th column of $U$, reduced modulo $d_i$ in the torsion coordinates;
\item[(iii)] with $S=\{\pm g_j:g_j\ne 0\}$, every $G$-path of length $\ell$
maps to a Cayley walk of length $\ell$, so
$d_{\Cay}(\phi(u),\phi(v))\le d_G(u,v)$: the only possible failure of
isometry is a shortcut.
\end{itemize}
\end{theorem}

\begin{proof}
(i) The relations imposed by Theorem~\ref{thm:cocycleZ} are exactly $Ag=0$;
the universal abelian group on generators $g_1,\dots,g_t$ subject to
$\Z$-linear relations $R$ is $\Z^t/\langle R\rangle$, and any other solution
defines a homomorphism $\Z^t\to\Gamma$ killing the row lattice, factoring
through the quotient. (ii) is the constructive structure theorem for
finitely generated abelian groups via the Smith normal form
\cite{stanley2016}: the unimodular column record $U$ is a change of generator
basis after which the relation lattice is diagonal. (iii) Consecutive labels
differ by $\pm g_{\mathrm{class}}\in S\cup\{0\}$; zeros only shorten the
walk.
\end{proof}

\begin{corollary}[The structural devices are SNF readouts]\label{cor:readout}
Under Theorem~\ref{thm:quotientZ}: (a) factor orders are computed, not
assigned: a class traversed $k$ times with consistent sign by some cycle and
crossed evenly by all others yields the relation $kg=0$, hence the torsion
factor $\Z_k$; (b) a generator whose $U$-column is supported on coordinates
$i_1,\dots,i_m$ is composite on exactly those dimensions; (c) the binary
theory is the case $2g=0$: appending the relations $2g_j=0$ reduces $A$
modulo $2$, and Theorem~\ref{thm:quotientZ} degenerates to
Theorem~\ref{thm:quotient2}.
\end{corollary}

\begin{proof}
(a),(b) are direct readings of the SNF data. (c) Adding rows $2e_j$ to the
relation lattice makes every coordinate $2$-torsion; the quotient is
$\F_2^t/\mathrm{rowspan}_{\F_2}(A\bmod 2)$, the $\F_2$ quotient.
\end{proof}

\begin{figure}[t]
\centering
\includegraphics[width=.98\linewidth]{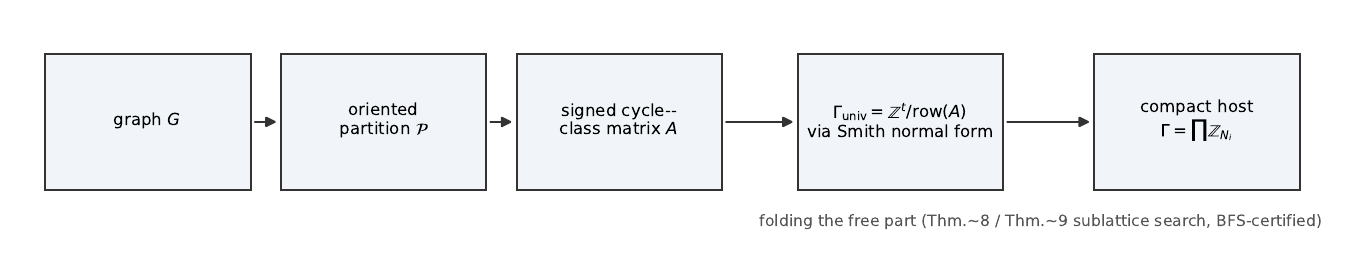}
\caption{The embedding pipeline. An oriented partition of $G$ yields a
signed cycle--class matrix $A$; the Smith normal form computes the universal
group $\Gamma_{\mathrm{univ}}=\Z^t/\mathrm{rowlattice}(A)$ and the
generators; folding the free part to a finite-index sublattice
(Section~\ref{sec:compact}) produces the compact host.}
\label{fig:pipeline}
\end{figure}

\begin{lemma}[$\Z$-Join Lemma]\label{lem:joinZ}
For the all-singleton oriented partition,
$\Gamma_{\mathrm{univ}}\cong\Z^{\,n-1}$ and the quotient labeling is
isometric on all pairs (with the infinite host
$\Cay(\Z^{\,n-1},S\cup-S)$).
\end{lemma}

\begin{proof}
With one class per oriented edge, $\Z^t=\Z^{E}$ is the (oriented) edge
module. The relation lattice is the integer cycle lattice
$\mathcal{Z}_\Z(G)=\ker\partial$, where
$\partial\colon\Z^{E}\to\Z^{V}_0$ is the boundary map sending an oriented
edge $u\to v$ to $\delta_v-\delta_u$, and $\Z^{V}_0$ denotes the sum-zero
sublattice of $\Z^{V}$, of rank $n-1$. Since $G$ is connected, $\partial$ is
surjective onto $\Z^{V}_0$, so
$\Gamma_{\mathrm{univ}}=\Z^{E}/\ker\partial\cong\Z^{V}_0\cong\Z^{\,n-1}$,
and under this isomorphism the label of $v$ is
$\phi(v)=\delta_v-\delta_r$, while the generator of the oriented edge
$u\to v$ is $\delta_v-\delta_u$.

Now let $\ell=d_{\Cay}(\phi(u),\phi(v))$ and let a geodesic word realize it;
summing its letters, the word is an integer vector $f\in\Z^{E}$ with
$\partial f=\delta_v-\delta_u$ and $\ell_1$-norm $\|f\|_1=\ell$ (each letter
contributes one unit to one oriented-edge coordinate). By the integral flow
decomposition theorem (see, e.g., \cite[Ch.~19]{schrijver1986}), any integer
vector $f$ with $\partial f=\delta_v-\delta_u$ decomposes as
$f=\chi_P+\sum_j c_j$, where $\chi_P$ is the signed characteristic vector of
a directed $u$--$v$ path and each $c_j$ is a signed circulation
($\partial c_j=0$), and the decomposition can be chosen
\emph{sign-coherent} with $f$: every nonzero coordinate of $\chi_P$ and of
each $c_j$ has the same sign as the corresponding coordinate of $f$.
Sign-coherence gives additivity of $\ell_1$-norms,
$\|f\|_1=\|\chi_P\|_1+\sum_j\|c_j\|_1\ge\|\chi_P\|_1=|P|\ge d_G(u,v)$.
Hence $\ell\ge d_G(u,v)$, and with Theorem~\ref{thm:quotientZ}(iii),
equality holds.
\end{proof}

\section{Compactification and existence}
\label{sec:compact}

$\Gamma_{\mathrm{univ}}$ may have free factors $\Z^f$; a finite host requires
a further quotient by a finite-index sublattice $L\subseteq\Z^f$. Folding is
itself an instance of Theorem~\ref{thm:quotientZ}: appending the rows of $L$
(pulled back to class coordinates) to $A$ and recomputing the SNF yields the
finite group $\Gamma_L$ and its labels. Quotients never stretch, so folding
preserves $d_{\Cay}\le d_G$ and can only introduce wraparound shortcuts.

\begin{theorem}[Sufficient diagonal fold]\label{thm:fold}
Let $R_i$ be the range of free coordinate $i$ over the vertex labels, i.e.\
$R_i=\max_v \phi(v)_i-\min_v\phi(v)_i$, and let $\gamma_i=\max_j|(g_j)_i|$.
The diagonal fold $L=N_1\Z\times\cdots\times N_f\Z$ is isometric whenever
$N_i>R_i+\diam(G)\,\gamma_i$ for every~$i$.
\end{theorem}

\begin{proof}
Suppose the folded labeling fails to be isometric for some pair $u,v$. By
Theorem~\ref{thm:quotientZ}(iii) applied to the folded group, the failure is
a shortcut: a word of length $\ell<d_G(u,v)\le\diam(G)$ joining the folded
labels. Lift the word letter by letter to $\Gamma_{\mathrm{univ}}$, starting
at $\phi(u)$: the lift ends at $\phi(v)+\lambda$ for some $\lambda\in L$. We
claim $\lambda=0$. Fix a free coordinate $i$ and write the $i$-th coordinate
of the lifted endpoint in two ways. On one hand, each of the $\ell$ letters
changes coordinate $i$ by at most $\gamma_i$ in absolute value, so
\[
\bigl|\,(\phi(v)+\lambda)_i-\phi(u)_i\,\bigr|\le\ell\,\gamma_i
\le\diam(G)\,\gamma_i.
\]
On the other hand, if $\lambda_i=q_iN_i$ with $q_i\ne 0$, then
\[
\bigl|\,(\phi(v)+\lambda)_i-\phi(u)_i\,\bigr|
\ge|q_i|N_i-\bigl|\phi(v)_i-\phi(u)_i\bigr|\ge N_i-R_i
>\diam(G)\,\gamma_i,
\]
using the hypothesis in the last step. The two displays contradict each
other, so $q_i=0$ for every $i$, i.e.\ $\lambda=0$. But then the lifted word
joins $\phi(u)$ to $\phi(v)$ in $\Gamma_{\mathrm{univ}}$ with length
$\ell<d_G(u,v)$, contradicting Lemma~\ref{lem:joinZ}. Hence no shortcut
exists and the fold is isometric.
\end{proof}

\begin{theorem}[Sublattice compactification; diagonal folds can miss the
optimum]\label{thm:sublattice}
The isometric finite quotients of a universal embedding are exactly the
finite-index sublattices $L\subseteq\Z^f$ whose folds pass the exact check,
and the minimal host over them is found by enumerating Hermite-normal-form
bases in increasing index, each verified by breadth-first search in the
finite Cayley graph. Restricting to diagonal $L$ can miss the optimum: for
the diamond graph the universal embedding has $f=2$ with labels
$(0,0),(1,0),(1,1),(2,1)$; \emph{every} diagonal fold of index $6$ fails,
yet the non-diagonal sublattice $L=\langle(3,0),(1,2)\rangle$ of index $6$
yields an isometric host of order $6$---the octahedron
$\Cay(\Z_2\times\Z_3,\{\pm(0,1),\pm(1,1)\})$.
\end{theorem}

\begin{proof}
Any finite abelian quotient of $\Gamma_{\mathrm{univ}}$ that is injective and
distance-preserving on the labels arises from a finite-index sublattice of
the free part; Hermite-normal-form bases enumerate sublattices by index
without repetition, and the per-candidate BFS check is exact (a wraparound
shortcut can only \emph{fail} the check, never falsely pass it).

The diamond assertions are finite verifications, which we report in full so
that they are reproducible. The universal labels above arise from the
oriented partition $\{0{\to}1,2{\to}3\}$, $\{0{\to}2,1{\to}3\}$,
$\{1{\to}2\}$ of the diamond with vertex $0$ opposite vertex $3$; the signed
cycle--class matrix for the two triangle cycles is
$\bigl(\begin{smallmatrix}1&-1&1\\-1&1&-1\end{smallmatrix}\bigr)$, of rank
$1$ with trivial invariant factor, so
$\Gamma_{\mathrm{univ}}\cong\Z^2$. The four diagonal sublattices of index
$6$---$\mathrm{diag}(1,6)$, $\mathrm{diag}(2,3)$, $\mathrm{diag}(3,2)$,
$\mathrm{diag}(6,1)$---each create a wraparound shortcut and fail the exact
check (all four verified computationally). Among all Hermite-normal-form
sublattices of index $6$, exactly two pass:
$\langle(3,0),(1,2)\rangle$ and $\langle(6,0),(5,1)\rangle$. For the first,
the folded labels are $(0,0),(0,2),(1,1),(1,0)$ in
$\Z^2/L\cong\Z_2\times\Z_3$, the edge differences generate
$S=\{\pm(0,1),\pm(1,1)\}$, and the resulting host
$\Cay(\Z_2\times\Z_3,S)$---the octahedron---is isometric on all six pairs:
the unique diamond non-edge $\{0,3\}$ maps to the antipodal pair
$(0,0),(1,0)$ at Cayley distance $2$, and the five edges map to adjacent
pairs. (Figure~\ref{fig:diamond}.)
\end{proof}

\begin{figure}[t]
\centering
\includegraphics[width=.9\linewidth]{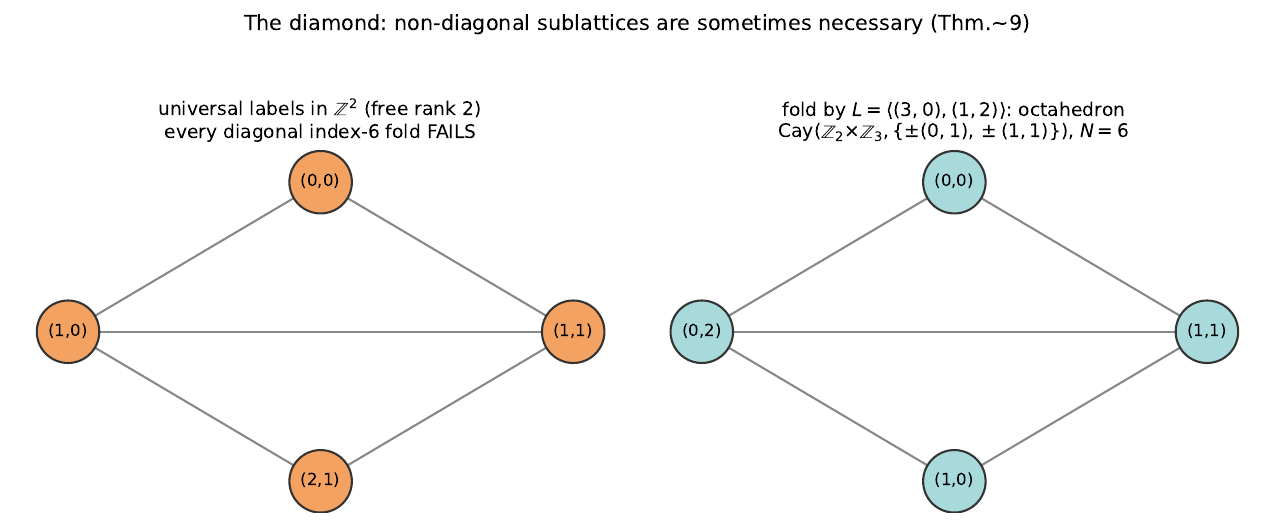}
\caption{The diamond, before and after the non-diagonal fold of
Theorem~\ref{thm:sublattice}. Left: universal labels in $\Z^2$; every
diagonal fold of index $6$ fails. Right: the fold by
$L=\langle(3,0),(1,2)\rangle$ gives the order-$6$ octahedron
$\Cay(\Z_2\times\Z_3,\{\pm(0,1),\pm(1,1)\})$, isometric on all pairs.}
\label{fig:diamond}
\end{figure}

\begin{theorem}[Existence, classical baseline]\label{thm:existence}
Every finite connected graph admits an isometric embedding into a Cayley
graph of a finite abelian group, with host order at most $2^{\,n-1}$.
\end{theorem}

\begin{proof}
Apply the binary quotient of the finest partition: by
Corollary~\ref{cor:naive} the labeling into $\Z_2^{\,n-1}$ is isometric, of
order $2^{\,n-1}$.
\end{proof}

\begin{remark}[What Theorem~\ref{thm:existence} is and is not]\label{rem:baseline}
The existence statement and the bound $2^{\,n-1}$ are not new: the
spanning-tree labeling is classical, going back to the earliest work on
Boolean-cube embeddings \cite{firsov1965}, and
Theorem~\ref{thm:existence} merely records that the quotient framework
re-derives it as its degenerate case. We state it for completeness because
it anchors the theory---every graph has \emph{some} abelian host, so
minimization is always well posed---but the contribution of this paper is
the machinery for finding hosts \emph{below} the baseline, not the baseline
itself. The companion paper \cite{fokam-p2} proves the bound is tight (odd
cycles) and quantifies how often it can be beaten.
\end{remark}

\subsection{A product rule}
\label{subsec:product}

Isometric abelian hosts compose across Cartesian products. The ingredients
of the following statement are classical---the distance in a Cartesian
product is the sum of the factor distances \cite{imrich-klavzar2000,
hammack2011}, and a product of abelian Cayley graphs is again an abelian
Cayley graph---but recording their combination gives a genuinely useful
construction tool, and it explains in one line the grid-into-torus
embeddings used at scale in the companion signal-processing papers
\cite{fokam-p3,fokam-p4}.

\begin{proposition}[Product rule]\label{prop:product}
For $i=1,\dots,d$ let $\phi_i\colon V(G_i)\to\Gamma_i$ be an isometric
embedding of the connected graph $G_i$ into $\Cay(\Gamma_i,S_i)$. Then
\[
\Cay(\Gamma_1,S_1)\,\square\,\cdots\,\square\,\Cay(\Gamma_d,S_d)
=\Cay\bigl(\Gamma_1\times\cdots\times\Gamma_d,\;S\bigr),
\qquad
S=\bigcup_{i=1}^{d}\{0\}^{i-1}\!\times S_i\times\{0\}^{d-i},
\]
and $\phi=(\phi_1,\dots,\phi_d)$ is an isometric embedding of the Cartesian
product $G_1\,\square\,\cdots\,\square\,G_d$ into
$\Cay(\prod_i\Gamma_i,S)$. Consequently, if $G$ embeds isometrically into a
Cartesian product $H_1\,\square\,\cdots\,\square\,H_d$ and each $H_i$
embeds isometrically into an abelian Cayley graph, then $G$ embeds
isometrically into an abelian Cayley graph of order $\prod_i|\Gamma_i|$.
\end{proposition}

\begin{proof}
The identification of the product of Cayley graphs is immediate: two
product vertices are adjacent iff they differ in exactly one coordinate by
a generator of that coordinate's factor, which is adjacency in
$\Cay(\prod\Gamma_i,S)$. The word metric of $S$ is additive across blocks:
each letter of $S$ moves exactly one coordinate, so a word for
$(a_1,\dots,a_d)$ partitions into subwords for each $a_i$, giving
$|(a_1,\dots,a_d)|_S\ge\sum_i|a_i|_{S_i}$, while concatenating factor
geodesics achieves equality. The distance in the Cartesian product of
graphs is likewise additive,
$d_{\square G_i}\bigl((u_i),(v_i)\bigr)=\sum_i d_{G_i}(u_i,v_i)$
\cite{imrich-klavzar2000}. Hence
\[
d_{\Cay}\bigl(\phi(u),\phi(v)\bigr)
=\sum_i\bigl|\phi_i(v_i)-\phi_i(u_i)\bigr|_{S_i}
=\sum_i d_{G_i}(u_i,v_i)
=d_{\square G_i}(u,v),
\]
using the isometry of each $\phi_i$ in the middle step. The final
statement is the composition of two isometric embeddings.
\end{proof}

\begin{corollary}[Submultiplicativity]\label{cor:product}
$\nu(G_1\,\square\,\cdots\,\square\,G_d)\le\prod_i\nu(G_i)$ and
$k_{\min}(G_1\,\square\,\cdots\,\square\,G_d)\le\sum_i k_{\min}(G_i)$.
In particular the $n\times n$ grid $P_n\,\square\,P_n$ embeds
isometrically into the torus $C_{2n-2}\,\square\,C_{2n-2}$ of order
$(2n-2)^2$, the construction used at image scale in
\cite{fokam-p3,fokam-p4}; and the length-$\ell$ genomic mutation space
$K_4^{\square\ell}=\Cay(\Z_4^\ell,S)$ embeds with $\nu=4^\ell=|V|$
(excursion ratio $1$ at every scale), a large non-grid host on which the
induced Fourier analysis is exact and canonical \cite{fokam-p3}.
\end{corollary}

\begin{remark}[A Graham--Winkler route, and two open questions]
\label{rem:gwroute}
Composed with the canonical isometric embedding of Graham and Winkler
\cite{graham-winkler1985}---which places any connected graph isometrically
in a Cartesian product of its quotient graphs---Proposition~\ref{prop:product}
yields an alternative pipeline: embed each canonical factor into an abelian
host, then take the product. Whether this route can beat the direct
quotient construction of Section~\ref{sec:quotient}, and whether the
bounds of Corollary~\ref{cor:product} are ever far from tight (is
$\nu$ multiplicative on products?), are natural open questions; the
certified value $\nu(P_4\,\square\,P_4)\le 36=\nu(P_4)^2$ from the
companion census \cite{fokam-p2} is consistent with equality but does not
prove it. A worked instance of the route is the Desargues graph, the
bipartite double cover of the Petersen graph and a partial cube: its
Graham--Winkler factors are quotients by $\theta^\ast$-cuts, none of which is a
path, cycle, or complete graph, and embedding each factor then multiplying
gives a compact abelian host---a concrete direction we leave to future work.
\end{remark}

\subsection{Worked examples}
\label{subsec:exhibit}

Table~\ref{tab:exhibit} collects three products, each with its host
\emph{certified isometric} by the exact check of
Section~\ref{sec:algorithm} (all four inter-factor claims were verified by
breadth-first search on the product host). They illustrate the three regimes
of Proposition~\ref{prop:product}: a product of non-cycle circulants that is
itself an abelian Cayley graph, hence optimal by
Theorem~\ref{thm:existence}'s equality companion \cite{fokam-p2}; a mixed
product combining the binary Petersen host of Example~\ref{ex:petersen} with
the cyclic host of $K_3$; and a product of two stars, whose factor host
$\Cay(\Z_6,\{1,3,5\})$ of order $6$ lies strictly below the binary star
dimension (a certified $\nu(K_{1,3})=6<2^{k_{\min}}=8$, an instance of the
compaction quantified in \cite{fokam-p2}).

\begin{table}[t]
\centering
\caption{Certified isometric embeddings of Cartesian products via
Proposition~\ref{prop:product}. Host orders marked $^{\ast}$ are optimal
($\nu$, by the abelian-Cayley equality of \cite{fokam-p2}); the others are
certified upper bounds, since $\nu$ of the Petersen factor is itself only
known to lie in $[11,16]$ \cite{fokam-p2}.}
\label{tab:exhibit}
\small
\begin{tabular}{lccc}
\toprule
product & $|V|$ & certified host (order) & vs.\ $2^{\,|V|-1}$\\
\midrule
$C_7(1,2)\,\square\,C_7(1,2)$ & $49$ & $\Z_7\times\Z_7$ ($49$)$^{\ast}$
 & $6\times10^{12}$\\
Petersen $\square\;K_3$ & $30$ & $\Z_2^4\times\Z_3$ ($48$) & $1\times10^{7}$\\
$K_{1,3}\,\square\,K_{1,3}$ & $16$ & $\Z_6\times\Z_6$ ($36$) & $9\times10^{2}$\\
\bottomrule
\end{tabular}
\end{table}

The first product is a $49$-vertex, $8$-regular abelian Cayley graph on which
no compaction is needed---the Product Rule reproduces the host exactly. The
second embeds a graph that is not itself a Cayley graph (Petersen) times a
cycle into a host of order $48$, seven orders of magnitude below the naive
binary bound. The third connects to the companion theory of stars: two claws,
each embedded in $\Cay(\Z_6,\{1,3,5\})$ rather than in $\Z_2^3$, product into
$\Z_6\times\Z_6$ of order $36$ rather than the $\Z_2^6$ of order $64$ that the
binary factors would give.

\section{The algorithm and its cost}
\label{sec:algorithm}

The construction is effective, and this section presents it in full: the
design, the pseudo-code, correctness, and a candid cost analysis.

\subsection{Design}
Algorithm~\ref{alg:embed} runs a \emph{portfolio} of initial oriented
partitions: structured constructors for cycles, paths, and complete graphs
(which propose the directed-cycle classes of
Remark~\ref{rem:matching}); a $\Phi$-guided four-cycle union--find; and a
chain-merge initializer filtered by the oriented $\Phi$-test of
Definition~\ref{def:Phi}. For each initial partition it computes the SNF
quotient (Theorem~\ref{thm:quotientZ}), searches folds in increasing host
order (Theorem~\ref{thm:sublattice}), and, when no fold verifies,
\emph{repairs} by peeling one edge off a largest class into a singleton and
re-quotienting. Since each repair strictly refines the partition, the loop
terminates at the all-singleton partition, whose binary quotient always
succeeds (Corollary~\ref{cor:naive}). Every candidate host is verified by
the exact check---all $\binom{n}{2}$ distances against a truncated BFS of
the candidate Cayley graph---so the returned embedding is certified
unconditionally. The portfolio design is motivated by
Remark~\ref{rem:phiposition}: $\varphi$- and $\Phi$-compatibility are
guides, not guarantees, so several starting points plus exact certification
are essential (the Pappus computation of Example~\ref{ex:pappus}, where only
$2$ of $15$ $\varphi$-compatible partitions are isometric, makes the point
concrete).

\begin{algorithm}[t]
\caption{Compact abelian embedding (portfolio + exact core + repair)}
\label{alg:embed}
\begin{algorithmic}[1]
\Require connected graph $G$
\Ensure certified isometric $\phi\colon V\to\Cay(\Gamma,S)$,
$\Gamma=\prod\Z_{N_i}$
\State compute all-pairs distances $D$ and $\diam(G)$
\State build initial oriented partitions (portfolio)
\ForAll{initial partitions $\mathcal{P}$}
  \While{true}
    \State SNF quotient $\to\Gamma_{\mathrm{univ}}$, generators, labels
      (Thm.~\ref{thm:quotientZ})
    \State fold search in increasing host order, each candidate checked
      exactly (Thm.~\ref{thm:sublattice})
    \If{verified} record $(\Gamma,S,\phi)$; \textbf{break}
    \EndIf
    \State \emph{repair:} peel one edge off a largest class into a singleton
    \If{all classes singleton} \textbf{break}
    \EndIf
  \EndWhile
\EndFor
\State binary terminal: finest-partition quotient
  (Thm.~\ref{thm:existence}), host $\le 2^{\,n-1}$
\State \Return smallest recorded verified host
\end{algorithmic}
\end{algorithm}

\begin{theorem}[Universality and termination]\label{thm:termination}
Algorithm~\ref{alg:embed} terminates on every connected graph and returns a
certified isometric embedding with $|\Gamma|\le 2^{\,n-1}$.
\end{theorem}

\begin{proof}
Each repair round strictly refines the partition, so each initializer runs at
most $m$ rounds; the binary terminal succeeds by
Theorem~\ref{thm:existence}. The returned embedding passed the exact check,
which compares all $\binom{n}{2}$ distances against a truncated BFS of the
candidate Cayley graph, so certification is unconditional.
\end{proof}

\subsection{Cost}
\label{subsec:cost}

We now account for the cost, and we are explicit about where it is
polynomial and where it is not.

\begin{theorem}[Cost of one pipeline pass]\label{thm:complexity}
Let $H$ be the host-order cap of the fold search, $f$ the free rank of
$\Gamma_{\mathrm{univ}}$, $k$ the binary dimension at the binary terminal,
and $R\le m$ the number of repair rounds. One full pass of
Algorithm~\ref{alg:embed} costs
\[
O\Bigl(\underbrace{n(n+m)}_{\text{distances}}
+\underbrace{m^2}_{\varphi/\Phi\text{ tests}}
+\;R\cdot\bigl[\underbrace{\mathrm{poly}(c,t)}_{\text{SNF}}
+\underbrace{\mathcal{N}_f(H)\,\bigl(H|S|+n^2\bigr)}_{\text{fold search}}
\bigr]
+\underbrace{2^{k}\,|S|+n^2}_{\text{binary terminal check}}\Bigr),
\]
where $\mathcal{N}_f(H)=\sum_{i\le H}\sigma_{f-1}^{*}(i)$ is the number of
sublattices of $\Z^f$ of index at most $H$ (with
$\sigma_{f-1}^{*}(i)=\sum_{d_1d_2\cdots d_f=i}d_2 d_3^2\cdots d_f^{f-1}$
counting Hermite-normal-form bases of index $i$).
\end{theorem}

\begin{proof}
Breadth-first search from every vertex gives the distance matrix in
$O(n(n+m))$. The $\varphi$ and $\Phi$ tests examine each of the
$\binom{m}{2}$ edge pairs in $O(1)$ time given the distance matrix. Each
repair round performs one Smith normal form on the $c\times t$ integer
cycle--class matrix, polynomial in its dimensions and entry sizes
\cite{schrijver1986}, and one fold search: the Hermite-normal-form bases of
index $i$ in $\Z^f$ are the upper-triangular integer matrices with diagonal
$d_1\cdots d_f=i$ and off-diagonal entries reduced modulo the diagonal,
counted by $\sigma^{*}_{f-1}(i)$ above; each candidate is checked by a BFS of
the folded Cayley graph truncated at depth $\diam(G)$, touching at most
$H$ group elements with $|S|$ generators each, followed by the
$\binom{n}{2}$-pair comparison. The binary terminal check is a BFS in
$\Cay(\Z_2^k,S)$ truncated at depth $\diam(G)$, which in the worst case
visits $\Theta(2^k)$ elements.
\end{proof}

\begin{remark}[Where the cost really lies]\label{rem:cost}
Three points deserve emphasis, because they bound what the algorithm can and
cannot promise. First, the \emph{certification} steps are the dominant
cost, and they are exponential in the worst case: the binary terminal check
is $\Theta(2^k)$ with $k$ up to $n-1$, and the fold search multiplies a
super-linear sublattice count $\mathcal{N}_f(H)$ (already
$\Theta(H^2)$ summed over indices for $f=2$, and growing rapidly with $f$)
by a per-candidate BFS. On the structured families that motivate the theory
(cycles, paths, grids, circulants, and the examples of
Section~\ref{sec:examples}) the quotient dimension $k$ and free rank $f$ are
small and the pipeline is fast in practice; on adversarial inputs it is not,
and we make no polynomial-time claim. Second, the certification cost is the
price of an unconditional guarantee: every returned host is exactly
verified, so the algorithm's outputs are theorems, not estimates. Third, we
do not know the complexity of the underlying optimization problem
(\emph{given $G$, compute $\nu(G)$ or $k_{\min}(G)$}); we conjecture it is
hard, but we have no reduction, and we state the question as open in
Section~\ref{sec:conclusion}.
\end{remark}

\section{Worked examples}
\label{sec:examples}

All numerical claims in this section were produced and re-verified by the
certified pipeline of Section~\ref{sec:algorithm}; the $\varphi$-classes and
parity matrices are displayed so that every computation can be reproduced by
hand or by machine.

\begin{example}[Triangle]\label{ex:triangle}
$K_3$ is not a partial cube. Its single independent cycle, traversed
cyclically with all three edges in one class oriented consistently, gives the
relation $3g=0$; by Theorem~\ref{thm:quotientZ} the universal group is
$\Z_3$ and the host is $\Cay(\Z_3,\{1,2\})=K_3$ itself, order $3$, isometric.
The matching constraint would have forced a binary host of order $4$
(Remark~\ref{rem:matching}).
\end{example}

\begin{example}[Petersen graph]\label{ex:petersen}
The Petersen graph has $n=10$, $m=15$, $c=6$. Computing the relation
$\varphi$ on all $\binom{15}{2}=105$ edge pairs yields exactly five classes
of size $3$ (Figure~\ref{fig:petersen}); with the standard vertex numbering
(outer $5$-cycle $0\ldots4$, inner pentagram $5\ldots9$, spokes $i\sim i+5$)
they are the five parallel matchings
\[
\begin{aligned}
F_1&=\{01,\,38,\,79\}, & F_2&=\{04,\,27,\,68\}, & F_3&=\{05,\,23,\,69\},\\
F_4&=\{12,\,49,\,58\}, & F_5&=\{16,\,34,\,57\}.
\end{aligned}
\]
So $t=5$, and on the standard cycle basis the $6\times 5$ cycle--class
parity matrix is
\[
A=\begin{pmatrix}
1&1&1&1&1\\ 1&1&1&1&1\\ 0&0&0&0&0\\ 1&1&1&1&1\\ 1&1&1&1&1\\ 0&0&0&0&0
\end{pmatrix},
\]
visibly of rank $\rho=1$: the unique relation is $g_1+\cdots+g_5=0$. Hence
$k=5-1=4$, with generators $e_1,e_2,e_3,e_4$ and the composite
$e_1+e_2+e_3+e_4=(1,1,1,1)$. The host is the Clebsch graph
$\Cay(\Z_2^4,\{e_1,e_2,e_3,e_4,(1,1,1,1)\})$ of order $16$, and the
embedding is certified isometric on all $\binom{10}{2}=45$ pairs. The
companion paper \cite{fokam-p2} shows $k=4=\lceil\log_2 10\rceil$ is
optimal. (The Petersen graph is vertex-transitive but not a Cayley graph
\cite{mckay-praeger1994}, so it cannot be its own host; the order-$16$
Clebsch host is remarkably close to the injectivity floor.)
\end{example}

\begin{figure}[t]
\centering
\includegraphics[width=.55\linewidth]{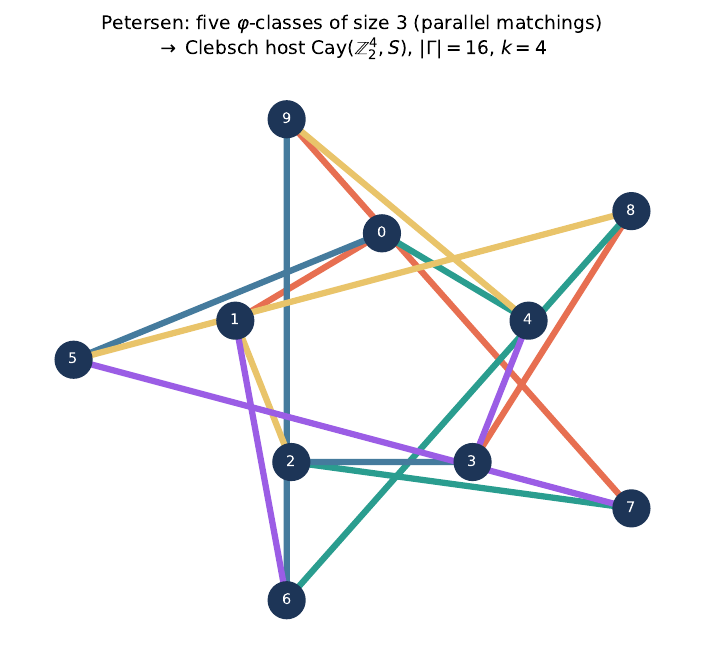}
\caption{The Petersen graph has five $\varphi$-classes of size three (the
parallel matchings, one per color). The cycle--class parity matrix has rank
$1$, so the binary quotient has dimension $k=4$: a certified isometric
embedding into the Clebsch graph $\Cay(\Z_2^4,S)$ of order $16$, with one
weight-four composite generator.}
\label{fig:petersen}
\end{figure}

\begin{example}[Pappus graph]\label{ex:pappus}
The Pappus graph ($n=18$, $m=27$, $c=10$; LCF notation $[5,7,-7,7,-7,-5]^3$)
admits exactly $27$ triples of pairwise $\varphi$-related edges, and exactly
$15$ partitions of $E$ into nine such triples. Of these $15$, precisely
\emph{two} yield isometric quotients---a concrete demonstration that
$\varphi$-compatibility does not imply embeddability
(Remark~\ref{rem:phiposition}) and that exact certification is
indispensable. One of the two (with vertices numbered along the LCF
Hamiltonian cycle) is
\[
\begin{aligned}
F_1&=\{(0,1),(6,7),(12,13)\}, & F_2&=\{(2,3),(8,9),(14,15)\},\\
F_3&=\{(0,5),(7,8),(15,16)\}, & F_4&=\{(3,4),(6,11),(13,14)\},\\
F_5&=\{(1,2),(9,10),(12,17)\}, & F_6&=\{(4,5),(10,11),(16,17)\},\\
F_7&=\{(0,17),(2,13),(4,15)\}, & F_8&=\{(1,8),(3,10),(5,6)\},\\
F_9&=\{(7,14),(9,16),(11,12)\},
\end{aligned}
\]
with $10\times 9$ cycle--class parity matrix (on the
\texttt{networkx} cycle basis)
\[
A=\begin{pmatrix}
1&1&1&1&1&1&0&0&0\\
0&0&0&0&0&0&0&0&0\\
0&0&1&1&1&0&1&1&1\\
1&1&0&0&0&1&1&1&1\\
1&1&0&0&0&1&1&1&1\\
0&0&0&0&0&0&0&0&0\\
0&0&0&0&0&0&0&0&0\\
1&1&1&1&1&1&0&0&0\\
0&0&1&1&1&0&1&1&1\\
0&0&1&1&1&0&1&1&1
\end{pmatrix},
\]
whose distinct nonzero rows are three, pairwise summing to each other, so
$\rho=2$ and $k=9-2=7$. The generator set has $|S|=9$: seven basis
generators and two composite generators, both of Hamming weight $5$
(namely $(1,0,0,1,1,1,1)$ and $(0,1,1,0,1,1,1)$ in the quotient
coordinates). The host has order $2^7=128$, against $2^{17}=131072$ for the
naive embedding---a $1024\times$ improvement---and isometry is certified on
all $\binom{18}{2}=153$ pairs.
\end{example}

\begin{example}[Diamond]\label{ex:diamond}
The diamond ($K_4$ minus an edge) illustrates the non-diagonal fold of
Theorem~\ref{thm:sublattice}, whose proof reports the full verified data:
universal labels $(0,0),(1,0),(1,1),(2,1)$ in $\Z^2$; all four diagonal
folds of index $6$ fail; the sublattice $\langle(3,0),(1,2)\rangle$ yields
the octahedron $\Cay(\Z_2\times\Z_3,\{\pm(0,1),\pm(1,1)\})$ of order $6$
with labels $(0,0),(0,2),(1,1),(1,0)$, certified isometric
(Figure~\ref{fig:diamond}). The order $6$ is optimal: the companion paper's
order bound gives $\nu\ge\max(n,2\,\diam)=\max(4,4)=4$, and an exhaustive
search over the three abelian groups of orders $4$ and $5$ confirms no
smaller isometric host exists.
\end{example}

\section{Concluding remarks}
\label{sec:conclusion}

We have given a uniform construction embedding any finite connected graph
isometrically into a Cayley graph of a finite abelian group, organized around
a single quotient labeling theorem whose binary and general forms are linked
by reduction modulo $2$. The relation $\varphi$ generalizes the
Djokovi\'c--Winkler relation beyond partial cubes while coinciding with it on
them; the partial-permutation constraint corrects the matching constraint of
the hypercube setting; the Smith normal form computes the most generic
consistent host exactly; and compactification is a further instance of the
same construction, with non-diagonal sublattices sometimes necessary. The
algorithmic pipeline certifies every output, at a certification cost that is
exponential in the worst case (Remark~\ref{rem:cost}).

Beyond the two open directions inherited from the examples---the exact
minimal order of the Petersen graph, and the complexity of computing
$\nu(G)$ or $k_{\min}(G)$---we single out the comparison with the
Graham--Winkler canonical embedding \cite{graham-winkler1985} sketched in the
introduction as the most natural structural question raised by this work.
Sharp lower bounds, exact dimensions for stars and odd cycles, and an
exhaustive census of all $995$ connected graphs on at most seven vertices
are developed in the companion paper \cite{fokam-p2}; the Fourier and
wavelet analysis supported by these embeddings is developed in two further
companion papers \cite{fokam-p3,fokam-p4}.

\section*{Data and code availability}
The verification pipeline and the scripts reproducing every computation in
Section~\ref{sec:examples} are available from the authors and will be
archived with the final version.

\section*{Declaration on the use of AI}
An AI assistant (Anthropic's Claude) was used for software development and
debugging of the verification pipeline and for language editing. All
mathematical content, directions, and conclusions are the authors' own; all
computational claims were verified by the certified pipeline of
Section~\ref{sec:algorithm}.

\section*{Acknowledgments}
The authors thank Solutum Engineering for internet and computing support,
and a referee of an earlier version for observations that led to
Remark~\ref{rem:gapfix} and to substantial expository improvements.

\end{document}